\documentclass[final,11pt,times]{elsarticle}
\usepackage{titlesec}
\titleformat{\paragraph}[block]{\normalfont\normalsize\bfseries}{\theparagraph}{1em}{}
\titlespacing*{\paragraph}{0pt}{1ex plus 0.5ex minus .2ex}{1em}

\emergencystretch=\maxdimen
\usepackage{algorithm}
\usepackage{algpseudocode}
\usepackage{amsmath,amssymb}
\usepackage{amsthm}

\newtheoremstyle{myplain} 
  {6pt} 
  {6pt} 
  {\itshape} 
  {} 
  {\bfseries} 
  {} 
  {0.5em} 
  {\thmname{#1}\ \thmnumber{#2}\ \thmnote{ (#3)}\newline} 

\theoremstyle{myplain}
\newtheorem{theorem}{Theorem}[section]

\newtheorem{corollary}[theorem]{Corollary}

\theoremstyle{remark}

\theoremstyle{definition}

\newtheorem{assumption}{Assumption}[section]

\usepackage{geometry}
\usepackage{fancyhdr}
\usepackage[table]{xcolor}
\usepackage{booktabs}
\usepackage{colortbl}
\usepackage{threeparttable}
\usepackage{xcolor}
\newcommand{\green}[1]{\textcolor{green}{#1}}

\usepackage{multirow}

\journal{CIE53 Proceedings}

\usepackage[colorlinks=true, pdfauthor={author}]{hyperref}

\begin{document}

\begin{frontmatter}


\title{A simulation-optimization approach for fractional, profitability-oriented inventory control under service-level type constraints}


\author[]{Tianxiao Sun, Ph.D\corref{cor}}
\ead{tianx.sun@gmail.com}
\cortext[cor]{Corresponding author} 
\author[]{Noah Schwarzkopf}
\ead{noahschwarzkopf98@gmail.com}
\address{Independent Researcher, Charlotte, NC, USA}
\begin{abstract}
Managing stock efficiently remains a core issue in modern logistics, where companies must reconcile cost efficiency with dependable service despite unpredictable market conditions. Conventional models often overlook the direct connection between investment in inventory and overall financial performance. This study introduces a data-driven decision framework that combines stochastic simulations with a profit-oriented optimization routine to enhance decision-making under uncertainty. The simulation stage generates performance estimates across multiple operating scenarios, providing realistic data on expenditures, revenues, and service reliability. These outcomes inform a fractional optimization process that searches for policies yielding the highest financial returns while maintaining required availability levels. The algorithm iteratively refines parameter values through feedback between simulated outcomes and optimization results, ensuring adaptability to dynamic enterprise systems. Computational experiments using representative business settings confirm that this approach improves both service consistency and financial yield. Overall, the framework demonstrates a practical, data-driven path for firms seeking to align operational responsiveness with sustainable profitability.
\end{abstract}


\end{frontmatter}
\thispagestyle{fancy}

\section{Introduction}
\label{sec:first}
Effective inventory control remains a central challenge in supply chain management, where companies must balance profitability and service performance under stochastic demand and operational uncertainty. Traditional inventory optimization models predominantly focus on cost minimization or service-level maximization, often treating financial performance as a secondary outcome. In practice, particularly in large retail and distribution environments, inventory decisions are increasingly evaluated using profitability-oriented metrics such as Gross Margin Return on Inventory Investment (GMROI), defined as the ratio of gross margin to average inventory investment, which directly measure how effectively inventory capital is converted into gross margin \cite{gokce2002optimization}.

Despite its managerial relevance, optimizing GMROI in realistic supply chain environments is challenging for two reasons. First, GMROI is a ratio-based metric, which gives rise to a fractional optimization problem. Second, in modern enterprise systems, key performance measures such as margin, inventory investment, and service levels are typically unavailable in closed form and must be estimated through stochastic simulation. This is particularly evident in large-scale complex distribution networks with heterogeneous products and stochastic demand, such as modern distribution centers (J Sainsbury's DC, shown on Figure~\ref{figure}).

\begin{figure}[!htbp]
\centering\includegraphics[width=0.7\linewidth]{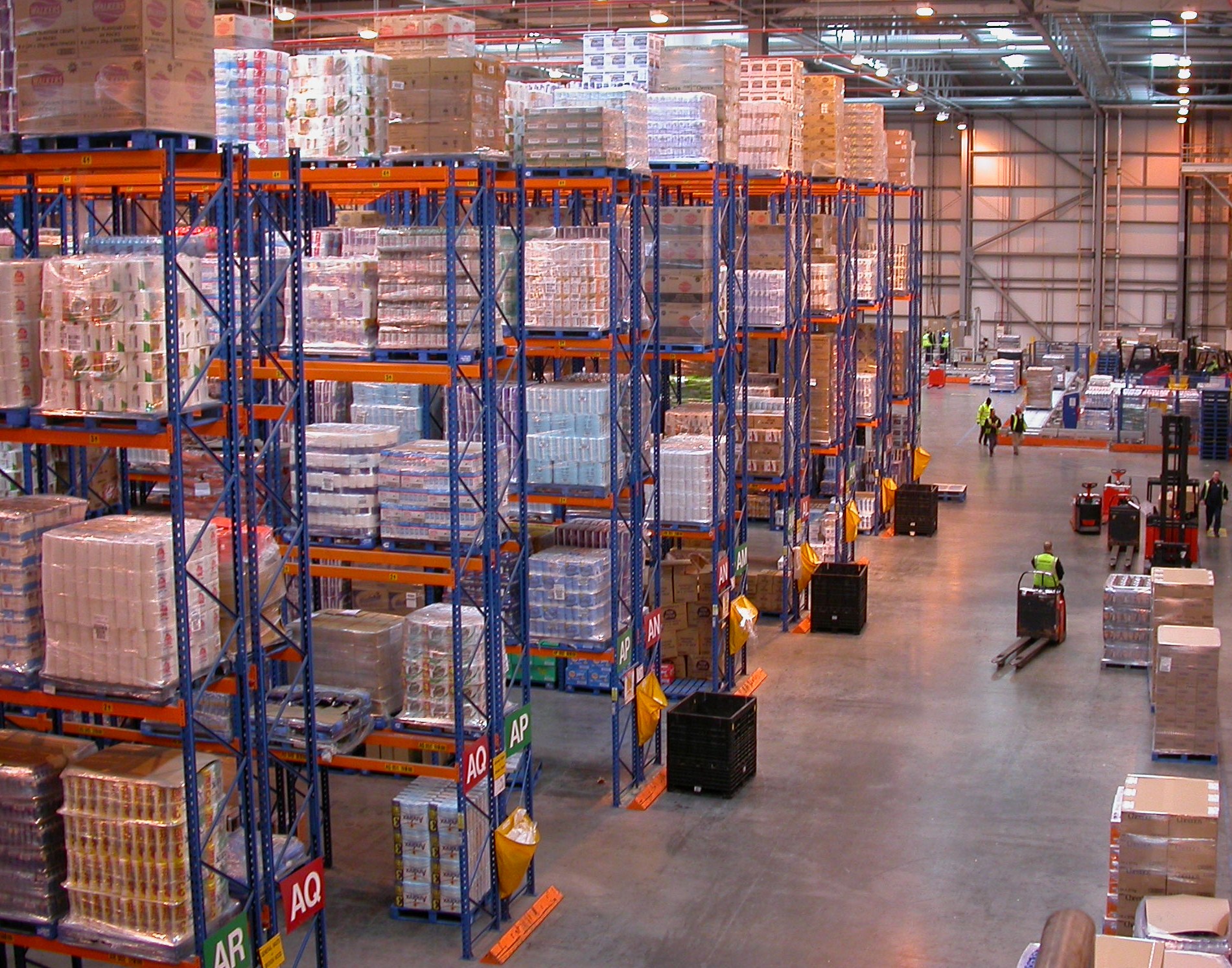}
\caption{Distribution centre (J Sainsbury's). 
Source: Nick Saltmarsh, via Wikimedia Commons, licensed under CC BY 2.0.}
\label{figure}
\end{figure}

Motivated by these challenges, this paper studies profitability-oriented inventory planning under service-level type constraints in simulation-driven enterprise environments. We consider a setting in which key performance measures are obtained through stochastic simulation and inventory decisions are evaluated using GMROI as the primary objective.

\subsection{Related Work}
Inventory optimization under uncertainty has been extensively studied in the operations research and supply chain literature. Classical analytical approaches address safety stock determination and service-level control using stochastic inventory theory, dynamic programming, robust optimization, and mixed-integer linear programming (MILP) \cite{smith2013handbook}\cite{bertsimas2006robust}. While these models provide strong theoretical guarantees, they rely on restrictive assumptions regarding demand distributions, lead times, and cost structures, limiting their applicability in large-scale, data-driven enterprise environments.

To overcome these limitations, simulation-based and hybrid simulation-optimization (sim-opt) approaches have gained increasing attention \cite{gonccalves2020operations}. Monte Carlo simulation is commonly used in practice to evaluate inventory policies under stochastic demand and to estimate nonlinear performance measures such as revenue, inventory investment, and service levels \cite{cruz2024integrated}\cite{vandeput2020inventory}. Early sim-opt studies embed heuristic or evolutionary search procedures within simulation to identify high-performing inventory policies \cite{avci2017multi}\cite{pitty2008decision}. More recent work integrates Monte Carlo simulation with modern black-box optimization techniques, such as Bayesian and hybrid optimization, to improve search efficiency under uncertainty \cite{cruz2024integrated}\cite{ogura2022bayesian}\cite{maitra2024inventory}.

Sim-opt often incurs substantial computational cost, as evaluating candidate policies typically requires repeated stochastic simulation runs when key performance measures such as service levels, revenues, and inventory investment are unavailable in closed form. This burden is particularly pronounced in large-scale enterprise settings with global coupling constraints and nonlinear performance measures. To improve scalability, metamodel-based and surrogate-assisted optimization methods have been proposed and shown to significantly reduce computational effort in safety stock optimization \cite{sharifnia2021robust}, with recent surveys confirming their growing maturity \cite{do2022metamodel}\cite{hong2021surrogate}.

Beyond computational considerations, simulation-based inventory optimization is further complicated by the presence of multiple, often conflicting objectives, such as cost efficiency and service performance. In the absence of a natural scalar objective, ranking-and-selection procedures \cite{tsai2017simulation} and Pareto-based evolutionary algorithms are commonly employed to explore trade-offs among competing objectives. While informative, such multi-objective approaches do not directly align the optimization objective with profitability-oriented managerial metrics. In particular, optimizing GMROI leads to a fractional objective that is not well represented by weighted-sum or Pareto-based formulations.

Such ratio-based objectives naturally fall within the scope of fractional programming. Foundational work by Charnes and Cooper \cite{charnes1962programming}, together with efficient solution methods such as Dinkelbach’s algorithm \cite{dinkelbach1967nonlinear}, establishes strong theoretical and algorithmic foundations for fractional optimization. Recent studies on benefit-cost and ratio optimization further underscore the importance of fractional objectives in large-scale decision-making contexts \cite{miller2022optimization}.

Overall, the literature reveals a clear gap at the intersection of simulation-based inventory optimization, profitability-oriented fractional objectives, and explicit service-level type constraints.
In particular, existing studies tend to address these dimensions in isolation, leaving the joint problem of simulation-driven performance evaluation, ratio-based profitability optimization, and service-level control largely unexplored.

\subsection{Our Contributions}
This paper makes the following contributions:
\begin{itemize}
    \item We develop a simulation-driven inventory planning framework in which Monte Carlo simulation generates scenario-based estimates of gross margin, inventory investment, and service performance for alternative inventory policies at the SKU level.
    
    \item We formulate inventory planning at the level of SKU groups (referred to as buckets) as a constrained fractional optimization problem that maximizes Gross Margin Return on Inventory Investment (GMROI) while enforcing service performance requirements expressed in terms of product availability, thereby aligning the optimization objective with profitability-oriented managerial metrics.
    
    \item We apply Dinkelbach’s algorithm~\cite{dinkelbach1967nonlinear} to the resulting discrete fractional program and show that, under the simulation-induced scenario structure, the algorithm converges in a finite number of iterations to a globally optimal solution.
    
    \item Through large-scale computational experiments using production-style data, we demonstrate that the proposed approach attains solutions that are numerically indistinguishable from the exact optimum, while delivering substantial computational advantages over standard integer optimization solvers and providing quantitative insight into service-profit trade-offs.
\end{itemize}

\section{Problem Formulation}
\label{sec:illustrations}
We consider an optimization framework defined over groups of SKUs, where each group represents a collection of products that are planned and evaluated jointly. Throughout the paper, we refer to such SKU groups as buckets. Inventory performance is evaluated through stochastic simulation at the SKU level, and optimization is performed over the resulting simulated scenarios aggregated to the group level. In this work, service performance is measured using in-stock percentage (ISP), which serves as the operational representation of service level. The objective is to maximize bucket-level GMROI subject to a minimum ISP requirement.

For each bucket, inventory and replenishment behavior is simulated at the SKU level over a future planning horizon under different safety stock settings. For a given SKU, each simulation scenario corresponds to operating the inventory system with a fixed safety stock level, where stochastic demand depletes inventory and replenishment orders are triggered according to a reorder point type policy in which the safety stock determines the reorder threshold. This process generates a characteristic sawtooth shaped inventory trajectory over time. From each simulated trajectory, gross margin, average inventory investment, and in-stock percentage are estimated at the SKU scenario level and then aggregated across SKUs to form the bucket-level GMROI objective and averaged in-stock percentage constraint used in the subsequent optimization. In detail, let
\begin{itemize}
    \item $M_{ij} \geq 0$ denote the gross margin dollars,
    \item $I_{ij} \geq 0$ denote the average inventory dollars,
    \item $S_{ij} \in [0, 1]$ denote the in-stock percentage (ISP), and
    \item $C_S \in [0, 1]$ denote the minimum required in-stock percentage goal at bucket level,
\end{itemize}
for SKU $i$, simulated scenario $j$, where $i \in \mathcal{I} = \{1, 2, \dots, n\}$ representing the SKU ID, and $j \in J_i$ is the index of the $j$th scenario for SKU $i$. Note that scenarios are dynamically chosen based on the demand and replenishment nature of SKUs.
We introduce binary decision variables
\[
x_{ij} =
\begin{cases}
1, & \text{if for SKU $i$ scenario $j$ is selected,}\\[4pt]
0, & \text{otherwise.}
\end{cases}
\]
Under a selected combination of scenarios, the resulting total gross margin is given by
\[
M({\mathbf{x}}) = \sum_{i\in\mathcal{I}}\sum_{j\in J_i} M_{ij}x_{ij},
\]
while the corresponding aggregate inventory investment is
\[
I({\mathbf{x}}) = \sum_{i\in\mathcal{I}}\sum_{j\in J_i} I_{ij}x_{ij},
\]
and the averaged in-stock percentage can be represented as
\[
S({\mathbf{x}}) = \frac{1}{n}\sum_{i\in\mathcal{I}}\sum_{j\in J_i} S_{ij}x_{ij},
\]
where $\mathbf{x}\in \mathcal{X}:=\{x_{ij} \in \{0, 1\}: \forall j\in J_i, \forall i\in \mathcal{I}\}$ lives in a heterogeneous product space with a total dimension $\sum_{i=1}^n|J_i|$. We want to solve the following optimization problem:
\begin{equation}
\tag{FP}
\label{prob:FP}
\begin{aligned}
\max_{\mathbf{x}} \quad 
& \frac{M(\mathbf{x})}{I(\mathbf{x})}  && \text{(Objective: GMROI at bucket level)} \\
\text{s.t.} \quad
& S(\mathbf{x}) \geq C_S, && \text{(Constraint: in-stock percentage goal)} \\
& \sum_{j\in J_i} x_{ij} = 1, \quad \forall i \in\mathcal{I}, && \text{(Constraint: regularization)} \\
& \mathbf{x} \in \prod_{i=1}^{n} \{0,1\}^{|J_i|}, && \text{(Feasible set: multi-dimensional binary space)}
\end{aligned}
\end{equation}

Furthermore, define the fractional objective
\[
F(\mathbf{x}) = \frac{M(\mathbf{x})}{I(\mathbf{x})}
\]
which corresponds exactly to the bucket-level GMROI; and the feasible region of the above program as
\[
\mathcal{D} = \{S(\mathbf{x})\geq C_S \text{ and } \sum_{j\in J_i} x_{ij} = 1,\ \forall i \in \mathcal{I} \mid \mathbf{x}\in \prod_{i=1}^{n} \{0,1\}^{|J_i|}\},
\]
then (\ref{prob:FP}) could be rewritten as $\max\{F(\mathbf{x}) \mid \mathbf{x}\in\mathcal{D}\}$.

With the above settings, we make the following assumptions in the rest of the paper:
\begin{assumption}
\textbf{(Positivity)} For all feasible selections $\mathbf{x} \in \mathcal{D}$, the total inventory investment satisfies $I(\mathbf{x}) > 0$, i.e., the aggregate inventory dollars are strictly positive for any feasible combination of scenarios.
\end{assumption}

\begin{assumption}
\label{feasibility}
\textbf{(Feasibility)}
The scenario sets generated in the simulation stage are constructed such that, for any prescribed in-stock threshold $C_S$, at least one feasible combination of SKU-level scenarios satisfies the aggregate service requirement.
\end{assumption}

\section{Algorithmic Framework to Optimize Constrained Fractional Binary Programming}
In this section, we introduce an iterative algorithmic framework to solve fractional programming (\ref{prob:FP}) and establish its convergence properties.

\subsection{Fractional Programming via Dinkelbach’s Algorithm}
We implement the Dinkelbach’s algorithm \cite{dinkelbach1967nonlinear} to solve the fractional programming (\ref{prob:FP}), with pseudocode provided in Algorithm~\ref{alg:dinkelbach}. Our decision variables are binary, and the feasible domain $\mathcal{D}$ is a finite set. Under this discrete setting, the convergence behavior of Dinkelbach's method exhibits finite-step termination rather than asymptotic convergence, and the solution obtained is globally optimal. This reduces the number of iterations needed to reach the optimum and offering notable efficiency gains for large-scale problem instances. A rigorous proof of finite-step termination and its theoretical bound is provided in Appendix~\ref{appendix:convergence}.

\begin{algorithm}[!htbp]
\caption{Dinkelbach's Algorithm for Maximizing GMROI}
\label{alg:dinkelbach}
\begin{algorithmic}[1]
  \State \texttt{INITIALIZE} $k \gets 0$ and $\lambda_k \gets 0$.
  \Repeat
    \State \texttt{SOLVE} the Subproblem
    \[
      W(\lambda_k) \;=\; \max_{\mathbf{x}\in \mathcal{D}}\{M(\mathbf{x}) - \lambda_k I(\mathbf{x})\},
    \]
    and denote the solution by $\mathbf{x}_k^\star$.
    \If{$W(\lambda_k) < \epsilon$}
      \State \texttt{STOP}; $\mathbf{x}_k^\star$ is the optimal solution and $\lambda^\star = \lambda_k^\star$ is the optimal ratio.
    \Else
      \State \texttt{UPDATE}
      \[
        \lambda_{k+1} \;=\; \frac{M(\mathbf{x}_k^\star)}{I(\mathbf{x}_k^\star)},\qquad k \gets k+1,
      \]
      \State \hskip\algorithmicindent and return to the \texttt{SOLVE} step.
    \EndIf
  \Until{convergence}
\end{algorithmic}
\end{algorithm}

The convergence tolerance, $\epsilon$, is defined as a small positive constant and is used to ensure numerical robustness in practice. Although Dinkelbach’s algorithm involves solving a sequence of linearized fractional subproblems in the \texttt{SOLVE} step, ensuring the computational tractability of each iteration remains essential for achieving efficient convergence in large-scale instances. In particular, since $W(\lambda)$ is strictly decreasing with a unique zero at the optimal ratio, the stopping condition $W(\lambda_k) < \epsilon$ serves as a numerically stable stopping criterion under finite-precision arithmetic.

\subsection{Binary Integer Programming subproblem}

At each iteration of Dinkelbach’s algorithm, the resulting subproblem is formulated as a \emph{binary integer linear program (BIP)}: the objective function is linear and all decision
variables are binary. The subproblem is solved exactly at each iteration.

From a computational perspective, this BIP is NP-hard, and solving it to optimality can be computationally demanding for large-scale instances with many SKUs and simulated scenarios. In our implementation, the subproblem is solved using standard integer optimization solvers through Python interfaces, including the CBC solver via \texttt{PuLP} and \texttt{OR-Tools}. We additionally tested GLPK and HiGHS through \texttt{PuLP} and SCIP through the \texttt{OR-Tools} interface; however, these alternatives exhibited inferior performance or reduced robustness on the instances considered. Consequently, CBC is adopted as the baseline exact solver for the computational experiments reported in this study.

Although modern integer solvers employ advanced branch-and-bound and cutting-plane techniques, the worst-case complexity of the subproblem remains exponential, which motivates the acceleration strategy introduced in the next subsection.

\subsection{Acceleration Strategy}
\label{subsec: acc}
To accelerate the computation of the inner subproblem $W(\lambda_k)$, we adopt a Lagrangian relaxation approach \cite{fisher1981lagrangian}, as shown in Algorithm~\ref{alg:lagrangian}. The main idea is to relax the coupling constraint $S(\mathbf{x}) \ge C_S$ by introducing a nonnegative Lagrange multiplier $\mu \ge 0$ and incorporating it into the objective function. For a fixed
$\lambda_k$, the relaxed subproblem can be written as
\begin{equation}
W_\mu(\lambda_k)
= \max_{\mathbf{x} \in \mathcal{X}}
\left\{
M(\mathbf{x}) - \lambda_k I(\mathbf{x}) - \mu \left( C_S - S(\mathbf{x}) \right)
\right\}.
\tag{1}
\end{equation}

Ignoring the constant term $-\mu C_S$ and recalling that
$
T(x) = \frac{1}{n} \sum_{i \in I} \sum_{j \in J_i} S_{ij} x_{ij},
$
the objective function can be expanded as
\begin{equation}
\max_{\mathbf{x} \in \mathcal{X}}
\sum_{i \in I} \sum_{j \in J_i}
\left(
M_{ij} - \lambda_k I_{ij} + \frac{\mu}{n} S_{ij}
\right) x_{ij}.
\tag{2}
\end{equation}

Define $W_{ij} = M_{ij} - \lambda_k I_{ij}$. Since each SKU $i$ is subject only to
the local constraint $\sum_{j \in J_i} x_{ij} = 1$, the relaxed problem is
separable across SKUs. In particular, for each SKU $i$, the optimal scenario
can be selected independently as
\begin{equation}
j^*(i)
= \arg\max_{j \in J_i}
\left(
W_{ij} + \frac{\mu}{n} S_{ij}
\right).
\tag{3}
\end{equation}
which reduces the computational complexity of the subproblem from a
binary integer program to a series of independent local decisions.

\begin{algorithm}[!htbp]
\caption{Lagrangian Relaxation for the Subproblem $W(\lambda_k)$}
\label{alg:lagrangian}
\begin{algorithmic}[1]
\Require $W_{ij} = M_{ij} - \lambda_k I_{ij}$, $S_{ij}$ for all $(i,j)$, threshold $C_S$
\Ensure Optimal solution $\mathbf{x}^*$ for the constrained Subproblem $W(\lambda_k)$

\Function{EnumerateSolve}{$\mu$}
    \For{each $i \in I$}
    \State $j^*(i) \gets
    \arg\max_{j \in J_i}
    \left( W_{ij} + \frac{\mu}{n} S_{ij} \right)$
    \EndFor
    \State \Return $\mathbf{x}(\mu) = \{ j^*(i) : i \in \mathcal{I} \}$
\EndFunction

\State \texttt{Initialize} $\mu_{\text{low}} \gets 0$, $\mu_{\text{high}} \gets 1$
\While{$S(\mathbf{x}(\mu_{\text{high}})) < C_S$}
    \State $\mathbf{x}(\mu_{\text{high}}) \gets$ \Call{EnumerateSolve}{$\mu_{\text{high}}$}
    \State $\mu_{\text{high}} \gets 2\,\mu_{\text{high}}$
\EndWhile

\While{$\mu_{\text{high}} - \mu_{\text{low}} > \varepsilon$}
    \State $\mu \gets (\mu_{\text{low}} + \mu_{\text{high}})/2$
    \State $\mathbf{x}(\mu) \gets$ \Call{EnumerateSolve}{$\mu$}
    \If{$S(\mathbf{x}(\mu)) \ge C_S$}
        \State $\mu_{\text{high}} \gets \mu$, \quad $\mathbf{x}^* \gets \mathbf{x}(\mu)$
    \Else
        \State $\mu_{\text{low}} \gets \mu$
    \EndIf
\EndWhile

\State \Return $\mathbf{x}^*,\, W(\mathbf{x}^*)$
\end{algorithmic}
\end{algorithm}

\paragraph{Optimality of the Subproblem}
The Lagrangian relaxation gives rise to a one-dimensional convex dual problem of the form
\[
\min_{\mu \ge 0} \max_{\mathbf{x} \in \mathcal{X}} \mathcal{L}(\mathbf{x},\mu).
\] Rather than solving this min--max problem explicitly, we exploit the monotonicity of the induced service level $S(\mathbf{x}(\mu))$ and recover the optimal multiplier efficiently
via binary search. For a fixed $\lambda_k$, consider the Lagrangian relaxation of the bucket-level service constraint. The associated dual function is
\[
\phi(\mu)
=
\max_{\mathbf{x} \in \mathcal{X}}
\left\{
M(\mathbf{x}) - \lambda_k I(\mathbf{x}) + \mu S(\mathbf{x})
\right\}
- \mu C_S,
\quad \mu \ge 0.
\]

For any fixed $\mu$, the inner maximization decomposes across SKUs, and the corresponding primal solution $x(\mu)$ is obtained by independently selecting $j^*(i)$ according to Algorithm~\ref{alg:lagrangian}. The dual function $\phi(\mu)$ is convex and piecewise linear in $\mu$, with subgradient
\[
\partial \phi(\mu) = \{ C_S - S(\mathbf{x}(\mu)) \}.
\]

An optimal multiplier $\mu^*$ satisfies the complementary slackness condition $\mu^*(C_S - S(\mathbf{x}^*)) = 0$, where $\mathbf{x}^* = \mathbf{x}(\mu^*)$. When the service constraint is binding, $\mu^* > 0$ and $S(\mathbf{x}^*) = C_S$; otherwise, the constraint is non-binding and $\mu^* = 0$. In either case, the optimal multiplier can be efficiently identified via one-dimensional search over $\mu$.

Although the decision variables are restricted to binary selections, the Lagrangian relaxation yields an exact solution for the relaxed subproblem. Any potential duality gap with respect to the original constrained binary integer program is discussed in the Remark of Section~\ref{remark on dual gap}.

\paragraph{Acceleration Effect}
Compared with directly solving $W(\lambda_k)$ via a general integer
programming solver (e.g., \texttt{PuLP} or \texttt{OR-Tools}), the Lagrangian relaxation exploits the separable structure of the problem. Each iteration only requires evaluating $\arg\max_{j \in J_i} \left( W_{ij} + \frac{\mu}{n} S_{ij} \right)$,  which has linear complexity in the number of SKUs. Moreover, since $\mu$ is a scalar variable, the binary search converges in $\mathcal{O}(\log(\mu_{\max}/\mu_{\min}))$ iterations. This substantially reduces the computational time while maintaining the same optimality guarantee for the subproblem.

\subsection{Degenerate Unconstrained Regime}
\label{subsec: unconstrained reg}
In practical supply chain applications, we observe that for certain SKU buckets and planning horizons, the simulated scenario sets are sufficiently rich relative to the prescribed service-level threshold. In such cases, the aggregate service constraint does not effectively restrict feasible decisions and the optimization problem degenerates into an unconstrained variant.

Formally, define
\[
\underline{S} \;=\; \frac{1}{n}\sum_{i\in I}\min_{j\in J_i} S_{ij}.
\]
When $\underline{S} \ge C_S$, the aggregate service-level constraint $S(\mathbf{x})\ge C_S$ is satisfied for all feasible selections $\mathbf{x}\in\mathcal{X}$. Consequently, the feasible region $\mathcal{D}=\{\mathbf{x}\in\mathcal{X}: S(\mathbf{x})\ge C_S\}$ coincides with $\mathcal{X}$, and the bucket-level fractional program reduces to an unconstrained GMROI maximization problem subject only to the per-SKU selection constraints. Although the fractional objective remains globally coupled across SKUs and the overall problem therefore remains jointly defined and non-decomposable, the Dinkelbach transformation resolves this coupling iteratively. As a result, each linearized subproblem $W(\lambda)=\max_{\mathbf{x}\in\mathcal{X}}\{M(\mathbf{x})-\lambda I(\mathbf{x})\}$ is separable across SKUs.

From a computational perspective, this structural simplification benefits all solution approaches. For exact integer solvers such as \texttt{PuLP} and \texttt{OR-Tools}, the removal of the aggregate service-level constraint yields tighter linear relaxations and significantly reduces branch-and-bound effort within each Dinkelbach iteration. For the proposed Lagrangian relaxation method, the uniformly non-binding regime corresponds to an optimal Lagrange multiplier $\mu^\star = 0$ derived from the complementary slackness condition. Under this condition, the relaxed subproblem reduces to a purely separable maximization across SKUs and admits a closed-form solution obtained by independently selecting, for each SKU, the scenario that maximizes $W_{ij}=M_{ij}-\lambda I_{ij}$. Algorithm~\ref{alg:unconstrained} summarizes this simplified solver, which can be interpreted as a degenerate specialization of Algorithm~\ref{alg:lagrangian}. The numerical experiments in Section~\ref{sec:numerical} explicitly report results from this regime and confirm the associated computational gains.

\begin{algorithm}[H]
\caption{Lagrangian Relaxation for the Subproblem $W(\lambda_k)$ (Unconstrained)}
\label{alg:unconstrained}
\begin{algorithmic}[1]
\Require $W_{ij}=M_{ij}-\lambda_k I_{ij}$ for all $(i,j)$
\Ensure Optimal selection $\mathbf{x}^\star$ for the unconstrained subproblem
\For{each SKU $i\in \mathcal{I}$}
    \State $j^\star(i)\leftarrow \arg\max_{j\in J_i} W_{ij}$
\EndFor
\State \Return $\mathbf{x}^\star=\{j^\star(i): i\in \mathcal{I}\}$
\end{algorithmic}
\end{algorithm}

\section{Numerical Experiments}
\label{sec:numerical}
This section evaluates the performance of the proposed bucket-level GMROI optimization framework on production-scale data. We compare three solution approaches: a binary integer programming (BIP) formulation solved via the \texttt{PuLP}
Python interface~\cite{Mitchell2011PuLPAL}, an equivalent formulation solved using \texttt{OR-Tools}~\cite{cpsatlp}, and the
proposed Lagrangian relaxation method integrated with the Dinkelbach outer loop. The experiments focus on solution quality, constraint satisfaction, and computational efficiency across multiple buckets of increasing scale.

\subsection{Experimental environment and data}
All experiments were conducted in a cloud-based Linux environment on an x86\_64 platform equipped with an 8-core 2.25 GHz processor. The purpose of the experimental setup is to evaluate solver behavior and scalability under controlled conditions, with relative performance comparisons across solution approaches as the primary focus.

The input data consist of representative synthetic inventory data calibrated to reflect the characteristics of a large-scale enterprise replenishment system. Prior to optimization, an isotonic regression procedure was applied to enforce monotonic relationships between safety stock levels and simulated performance metrics. In particular, gross margin, average inventory investment, and in-stock percentage are expected to be non-decreasing as safety stock increases; isotonic regression is used to smooth stochastic noise while preserving this structural relationship.

The experimental instances follow the bucket-level optimization framework described in Section~\ref{sec:illustrations}. For each bucket, scenario-level performance metrics are precomputed and used as inputs to the GMROI maximization problem subject to an in-stock percentage constraint.

\subsection{Buckets and solver configurations}
\label{subsec: config}
A set of representative SKU buckets was selected for evaluation in each experimental regime, spanning nearly two orders of magnitude in problem size. The smallest bucket contains 3,944 SKUs and 234,160 simulated scenarios, while the largest bucket includes 91,155 SKUs and more than five million scenarios.

For each bucket and each target regime, all solution approaches were initialized with identical input data, in-stock percentage goals, and convergence tolerances. The fractional GMROI objective was handled consistently using the Dinkelbach procedure, and exact integer solutions were obtained using \texttt{PuLP} and \texttt{OR-Tools}, while the proposed Lagrangian relaxation was applied as described in Section~\ref{subsec: acc}. This experimental design isolates the impact of problem structure on solver performance and facilitates a controlled comparison across buckets and regimes.

The ISP goals were chosen deliberately to examine both constrained and unconstrained regimes of the bucket-level optimization problem. Specifically, we considered two classes of ISP goals. First, for each bucket, we selected a potentially binding ISP goals, defined as the midpoint between the empirically observed lower and upper bounds of achievable bucket-level ISP across all feasible scenario selections. This construction ensures that the service constraint is neither uniformly non-binding nor trivially infeasible, and that the effective degree of constraint tightness is
comparable across buckets of different scales. As a result, the constrained instances reflect a consistent binding intensity, enabling a fair comparison of solver behavior and computational efficiency. Second, we additionally evaluated unconstrained instances by selecting ISP goals that lie below the achievable lower bound, rendering the aggregate service constraint uniformly non-binding. These instances serve to illustrate the degenerate regime discussed in Section~\ref{subsec: unconstrained reg}, in which the optimization problem reduces to an unconstrained GMROI maximization and all solvers benefit from structural simplification.

In addition to the baseline configurations, we conducted supplementary experiments to examine the effect of the ISP constraint tightness on computational performance. For each bucket, the ISP goal was systematically varied within the achievable range to induce different degrees of constraint binding, from weakly operative to strongly restrictive regimes. This experimental setup is designed to evaluate how solution time scales with increasing constraint severity across different solution approaches.

\subsection{Results}
This section summarizes the optimization outcomes and computational performance across all buckets and solution approaches. We first report baseline results under the primary in-stock configurations, followed by additional analyses that isolate the constrained regime and examine the effect of service-level tightness on solver runtime. Throughout this section, solution quality is assessed using the metric $\mathrm{TAR\_ERR}$, which measures the relative deviation of a realized solution from the exact GMROI optimum and is defined as
\[
\mathrm{TAR\_ERR}
=
\frac{\lvert F(\mathbf{x^{\star}_{\text{exact}}}) - F(\mathbf{x^{\star}_{\text{solver}}}) \rvert}
{\max\{1, \lvert F(\mathbf{x^{\star}_{\text{exact}}}) \rvert\}}.
\]
Here, $\mathbf{x^{\star}_{\text{exact}}}$ denotes the exact GMROI-optimal solution obtained using \texttt{PuLP}, while $\mathbf{x^{\star}_{\text{solver}}}$ denotes the solution returned by the solver under evaluation. Throughout this paper, solutions obtained from \texttt{PuLP} and \texttt{OR-Tools} are treated as exact, as both solvers consistently return numerically identical GMROI values under the prescribed convergence tolerance. When a table entry is reported as \texttt{EXACT}, the corresponding optimal value is attained at the solver level and is numerically indistinguishable from the true optimum; in such cases we observe $\mathrm{TAR\_ERR} \le 10^{-14}$.

\begin{table}[htbp]
\centering
\small
\caption{Optimization results and runtimes}
\label{tab:opt_results}

\begin{threeparttable}
\begin{tabular}{@{} l
   r r
   r r
   r r r
   @{}}
\toprule
Type\#SKU-\#SCEN

& \multicolumn{2}{c@{}}{\makebox[2.5cm][c]{\texttt{PuLP-CBC}}}
& \multicolumn{2}{c@{}}{\makebox[2.5cm][c]{\texttt{OR-Tools-CBC}}}
& \multicolumn{3}{c@{}}{\makebox[3.0cm][c]{Lagrangian Relaxation}}\\
\cmidrule(lr){2-3}\cmidrule(lr){4-5}\cmidrule(lr){6-8}
& \#ITER & t[s]
& \#ITER & t[s]
& \#ITER & t[s] & $\mathrm{TAR\_ERR}$\\
\midrule

\multicolumn{8}{c}{\textbf{constrained}} \\
\midrule

C03944-0234160
& 3 & 118
& 4 & 82
& 3 & \green{0.184}  & 3.2e-07\\

C07158-0421147
& 3 & 427
& 4 & 366
& 3 & \green{0.389}  & 8.5e-06\\

C11045-0650493
& 3 & 497
& 4 & 600
& 3 & \green{0.621}  & 2.0e-06\\

C34280-1975791
& 3 & 1,876
& 4 & 2,885
& 4 & \green{3.09}  & 6.6e-07\\

C40533-2178351
& 4 & 3,961
& 4 & 3.060
& 4 & \green{3.3}  & \texttt{EXACT}\\

C47200-2644766
& 4 & 9,627
& 4 & 8,874
& 4 & \green{4.16}  & 1.4e-06\\

C59813-3242787
& 3 & 10,495
& - & N/A$^{\ast}$
& 4 & \green{5.09}  & 1.4e-06\\

C71925-3762114
& 4 & 18,637
& 4 & 12,834
& 4 & \green{5.77}  & 2.7e-07\\

C91155-5035313
& 4 & 27,119
& - & N/A$^{\ast}$
& 4 & \green{9.43}  & 1.2e-08\\

\midrule
\multicolumn{8}{c}{\textbf{unconstrained}} \\
\midrule

U03944-0234160
& 3 & 33.7
& 3 & 7.47
& 3 & \green{0.051}  & \texttt{EXACT}\\

U07158-0421147
& 3 & 63
& 3 & 16.7
& 3 & \green{0.102}  & \texttt{EXACT}\\

U11045-0650493
& 3 & 95
& 3 & 22.6
& 3 & \green{0.151}  & \texttt{EXACT}\\

U34280-1975791
& 4 & 318
& 4 & 85
& 4 & \green{0.534}  & \texttt{EXACT}\\

U40533-2178351
& 4 & 324
& 4 & 90
& 4 & \green{0.614}  & \texttt{EXACT}\\

U47200-2644766
& 3 & 438
& 3 & 129
& 3 & \green{0.731}  & \texttt{EXACT}\\

U59813-3242787
& 4 & 481
& 4 & 135
& 4 & \green{0.913}  & \texttt{EXACT}\\

U71925-3762114
& 4 & 890
& 4 & 277
& 4 & \green{1.01}  & \texttt{EXACT}\\

U91155-5035313
& 4 & 1,211
& 4 & 406
& 4 & \green{1.48}  & \texttt{EXACT}\\
\bottomrule
\end{tabular}

\begin{tablenotes}
\footnotesize
\item[$\ast$] Run exceeded the predefined computational limits under the experimental environment.
\end{tablenotes}
\end{threeparttable}
\end{table}

\begin{figure}[t]
  \centering
  \includegraphics[width=0.95\linewidth]{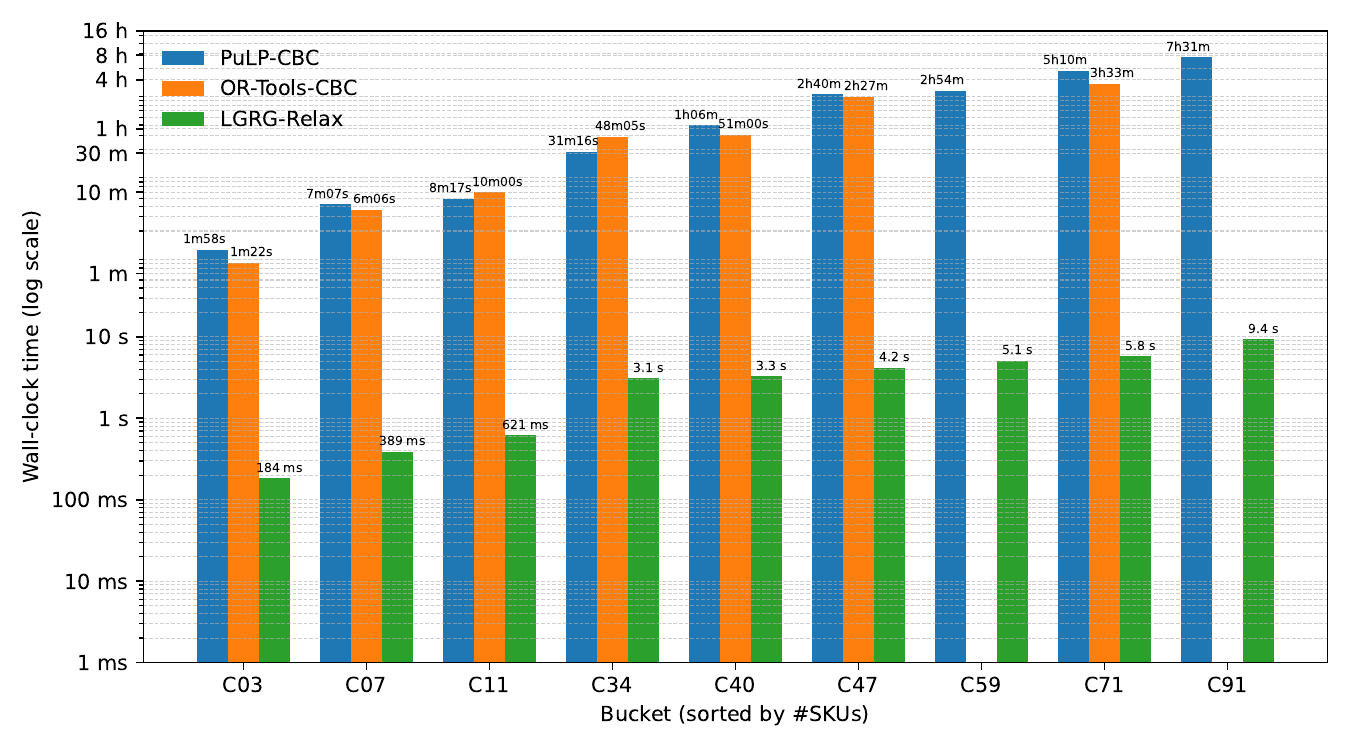}
  \caption{Wall-clock time across solvers in the constrained regime (log scale)}
  \label{fig:constrained_results}
\end{figure}

\paragraph{Baseline Results}
The constrained part of Table~\ref{tab:opt_results} summarizes the bucket-level optimization baseline outcomes across all solution
approaches under the primary in-stock configurations described in Section~\ref{subsec: config}. For each bucket and solver, we report the number of SKUs and simulated scenarios, the number of Dinkelbach iterations required for convergence, the achieved GMROI, the realized in-stock percentage (ISP), and the overall wall-clock time. Across all buckets, the proposed Lagrangian-based method attains GMROI values that are numerically indistinguishable from those obtained by exact integer solvers, while consistently satisfying the prescribed service-level requirements. The maximum observed deviation in GMROI relative to the \texttt{PuLP} baseline is on or under the order of $10^{-5}$, which is negligible at production scale. 

In addition to solution quality, substantial differences are observed in computational efficiency. As illustrated in Figure~\ref{fig:constrained_results}, the wall-clock time required by the Lagrangian relaxation method is orders of magnitude smaller than that of the exact integer solvers on larger buckets. These performance differences become increasingly pronounced as problem size grows, motivating the detailed computational analysis presented in the next subsection.

\paragraph{Unconstrained Regime}

\begin{figure}[ht]
  \centering
  \includegraphics[width=0.95\linewidth]{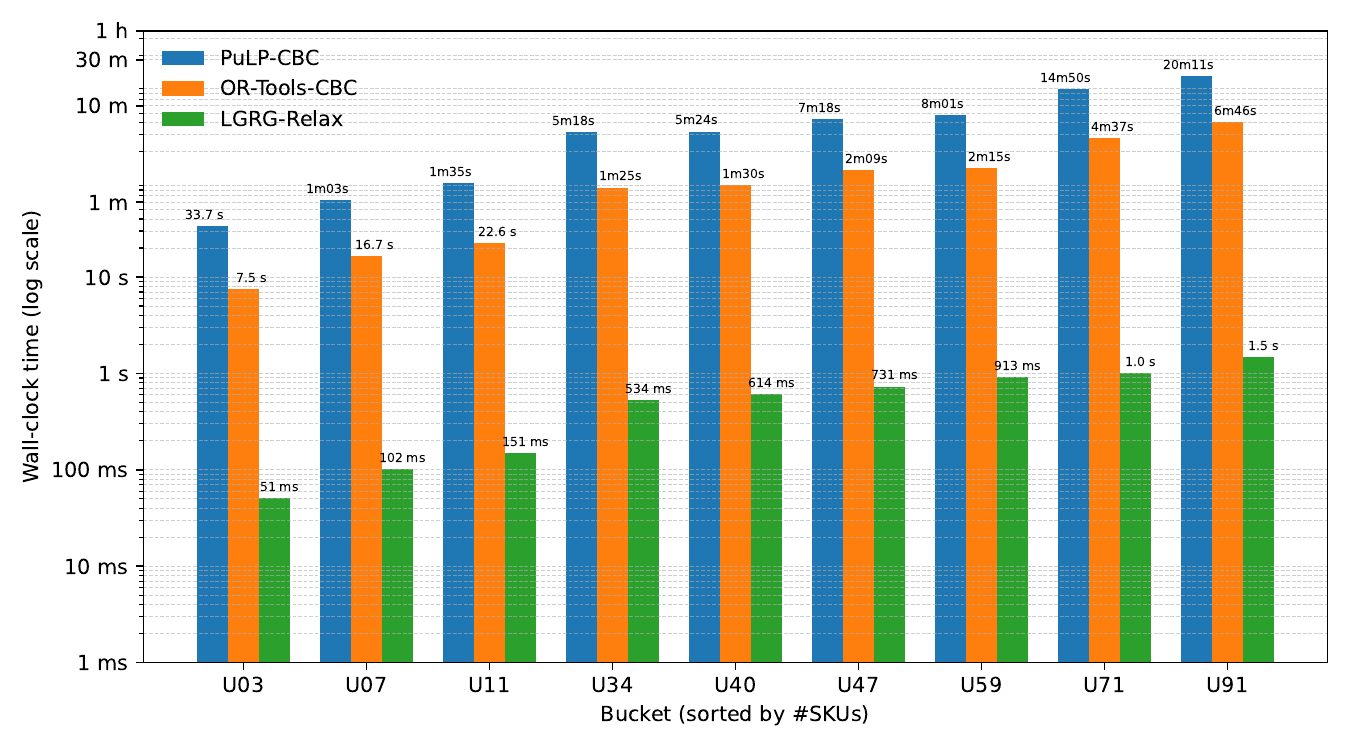}
  \caption{Wall-clock time across solvers in the unconstrained regime (log scale)}
  \label{fig:unconstrained_results}
\end{figure}

In the uniformly non-binding regime described in Section~\ref{subsec: unconstrained reg}, the optimization problem reduces to an unconstrained GMROI maximization. As reported in
the unconstrained part of Table~\ref{tab:opt_results} and shown in Figure~\ref{fig:unconstrained_results}, all solvers benefit from the removal of the aggregate coupling constraint, exhibiting substantially reduced solution times relative to their constrained counterparts, while
producing numerically identical GMROI-optimal selections, and among all solvers, the Lagrangian-based approach still achieves the lowest wall-clock times.

\paragraph{Effect of Constraint Tightness}
Figure~\ref{fig:tightness_scaling} examines the effect of constraint tightness using a medium-sized bucket (C11/U11) as a representative case. The numeric annotations shown next to each data point indicate the number of Dinkelbach iterations required for convergence under the corresponding constraint tightness. For confidentiality, constraint tightness is reported on a normalized relative scale, which preserves the relative ordering of constraint severity while abstracting away absolute service-level values. The in-stock requirement is varied across its empirically feasible range on a normalized relative scale, from an effectively unconstrained regime to a strongly binding upper limit (near 100\%). As the constraint is tightened, the realized GMROI declines relative to the unconstrained optimum. This loss is initially modest but becomes increasingly pronounced as the constraint approaches its most restrictive regime, highlighting the increasing marginal economic cost associated with more restrictive service-level requirements.

Each subplot in Figure~\ref{fig:tightness_scaling} illustrates the relationship between constraint tightness and runtime for a specific solution approach. For both the Lagrangian relaxation and the \texttt{PuLP} solver, runtime exhibits a clearly non-monotonic pattern: the longest solution times occur at intermediate levels of constraint tightness, while both the unconstrained regime and the fully binding regime are comparatively easier to solve. This behavior is consistent with the discrete nature of the problem. In the unconstrained regime, the service constraint is inactive and the problem simplifies structurally, while near the binding upper limit the feasible region becomes highly restricted, often yielding fewer admissible solutions. In contrast, intermediate targets induce the greatest combinatorial ambiguity, with multiple near-feasible and near-optimal selections competing, which is also reflected in a higher number of Dinkelbach iterations (e.g., 4 iterations in the mid-range compared with 2-3 near the extremes).

Comparing the two panels of Figure~\ref{fig:tightness_scaling} further highlights the computational advantage of the proposed Lagrangian relaxation. Across all levels of constraint tightness, the Lagrangian-based method is consistently and substantially faster than the exact \texttt{PuLP} solver, often by several orders of magnitude in wall-clock time. Notably, this speedup is achieved without sacrificing solution quality, as both solvers exhibit nearly identical relative GMROI trends across the entire tightness spectrum.

\begin{figure}[t]
  \centering
  \includegraphics[width=0.99\linewidth]{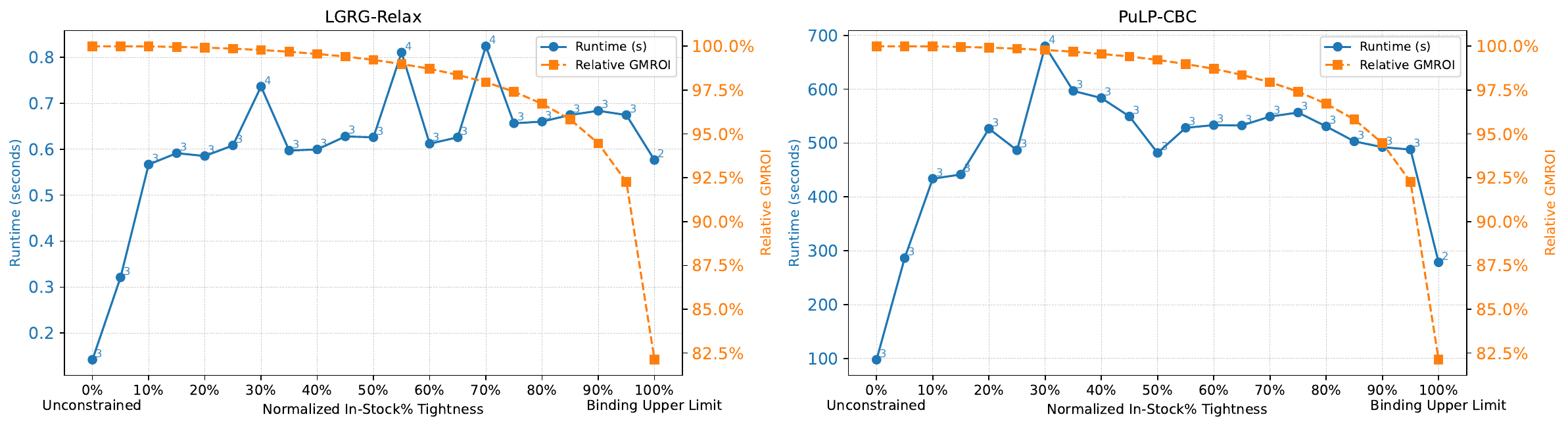}
  \caption{Runtime and Relative GMROI vs Normalized In-Stock Tightness}
  \label{fig:tightness_scaling}
\end{figure}

\subsection{Computational performance}
The computational efficiency of the proposed Lagrangian--Dinkelbach framework is primarily driven by its structural decomposition of the inner optimization problem. Solving the BIP subproblem directly requires global branch-and-bound over a large discrete decision space, which is computationally expensive.

The Lagrangian relaxation removes the bucket-level service constraint and decomposes the subproblem into independent SKU-level decisions for a fixed multiplier. As a result, each inner iteration reduces to a set of simple local maximization operations rather than a combinatorial search. Moreover, the Lagrangian multiplier is one-dimensional and can be efficiently updated via binary search, while the Dinkelbach outer loop converges in a small number of iterations. Together, these properties significantly reduce the overall computational burden while preserving solution quality.

\paragraph{Remark on Duality Gap in the Integer Setting}
\label{remark on dual gap}
Algorithm~\ref{alg:lagrangian} relies on a Lagrangian relaxation of the bucket-level service constraint and is exact for the associated continuous relaxation. When the decision variables are restricted to multiple-choice binary selections, a nonzero duality gap may in principle arise because the service-level constraint couples otherwise separable SKU-level decisions. Such linking constraints are
well known to induce duality gaps in integer programs.

Consistent with this theory, the Lagrangian solution is not guaranteed to be globally optimal for the original binary integer program. However, the numerical experiments demonstrate that the resulting gap is negligible in practice: across all buckets, the Lagrangian method attains GMROI values that are numerically indistinguishable from those produced by exact integer optimization solvers, while satisfying the ISP constraints. This behavior suggests that, for the bucket structures encountered in
production, the Lagrangian relaxation is extremely tight and provides an effective approximation to the true integer optimum with substantial computational savings.

\section{Conclusion}
\label{sec:conclusion}
This paper proposes a scalable simulation--optimization framework for bucket-level GMROI maximization with discrete SKU decisions under a global service performance constraint (operationalized as an in-stock percentage requirement). By combining stochastic simulation with a Dinkelbach-based fractional programming formulation, the approach enables profitability-oriented decision-making while explicitly accounting for service performance requirements. The proposed Lagrangian relaxation further exploits the separable structure of the problem, yielding substantial computational gains without compromising solution quality.

Beyond demonstrating algorithmic efficiency, our results provide managerial insight into the role of constraint tightness. By systematically varying the in-stock percentage (ISP) within its empirically feasible range, we show that increasingly stringent ISP requirements lead to a growing degradation in realized GMROI. While moderate ISP targets impose only limited economic penalties, aggressively high ISP requirements incur disproportionately large profitability losses. This finding underscores the necessity for supply chain practitioners to carefully balance service ambitions against financial returns, rather than treating service targets, implemented through ISP requirements, as exogenous or cost-free design parameters.

Taken together, the results illustrate both the operational value and the practical relevance of the proposed framework. From an optimization perspective, the method offers a tractable and scalable solution to large-scale, service-constrained fractional programs. From a supply chain management perspective, it provides a quantitative basis for evaluating service--profit trade-offs, supporting informed and economically grounded inventory decisions in complex, data-driven enterprise environments.

\section*{Acknowledgments}
This research was conducted using data, resources, and project support provided by an industry partner. The authors sincerely thank the partner for enabling this study and for the insights and collaboration that made this work possible.


\appendix
\renewcommand{\thesection}{\Alph{section}}

\titleformat{\section}
  {\normalfont\Large\bfseries} 
  {Appendix \thesection}       
  {1em}                        
  {}                           
  []

\renewcommand{\thetheorem}{\Alph{section}.\arabic{theorem}}
\renewcommand{\theequation}{\Alph{section}.\arabic{equation}}
\renewcommand{\thefigure}{\Alph{section}.\arabic{figure}}

\section{Convergence and Finite-Step Termination Proofs}
\label{appendix:convergence}

\begin{theorem}[General Convergence of Dinkelbach's Algorithm]
\label{thm:convergence}
Consider the fractional programming problem
\[
\max_{\mathbf{x} \in \mathcal{D}} 
F(\mathbf{x}) = \frac{M(\mathbf{x})}{I(\mathbf{x})},
\]
where \(\mathcal{D}\) is compact and nonempty, and $I(\mathbf{x})$ satisfies $I(\mathbf{x})>0$ for all $\mathbf{x}\in\mathcal{D}$. Assume further that $I$ is lower semicontinuous on $\mathcal{D}$.
Define
\[
\Phi(\lambda) = \max_{\mathbf{x} \in \mathcal{D}} 
\{ M(\mathbf{x}) - \lambda I(\mathbf{x}) \},
\qquad
\lambda^\star = \max_{\mathbf{x} \in \mathcal{D}} 
\frac{M(\mathbf{x})}{I(\mathbf{x})}.
\]
If each subproblem of the Dinkelbach iteration
\[
\mathbf{x}_k \in 
\arg\max_{\mathbf{x} \in \mathcal{D}} 
\{ M(\mathbf{x}) - \lambda_k I(\mathbf{x}) \},
\qquad
\lambda_{k+1} = \frac{M(\mathbf{x}_k)}{I(\mathbf{x}_k)},
\]
is solved exactly, then the sequence \(\{\lambda_k\}\) satisfies:
\begin{enumerate}
\item[\textnormal{(a)}] \(\{\lambda_k\}\) is monotonically increasing and bounded above by \(\lambda^\star\);
\item[\textnormal{(b)}] \(\lambda_k \to \lambda^\star\) as \(k \to \infty\);
\item[\textnormal{(c)}] \(\mathbf{x}_k \to \mathbf{x}^\star \in \arg\max_{\mathbf{x}\in \mathcal{D}} F(\mathbf{x})\).
\end{enumerate}
\end{theorem}

\begin{proof}
From the iteration rule,
\[
\lambda_{k+1} - \lambda_k
= \frac{M(\mathbf{x}_k)}{I(\mathbf{x}_k)} - \lambda_k
= \frac{\Phi(\lambda_k)}{I(\mathbf{x}_k)}.
\]
For each \(\lambda \in \mathbb{R}\), the function
\[
\Phi(\lambda) = \max_{\mathbf{x} \in \mathcal{D}} 
\{ M(\mathbf{x}) - \lambda I(\mathbf{x}) \}
\]
is continuous and strictly decreasing, as it is the pointwise maximum of affine functions in \(\lambda\).

Therefore, whenever $\Phi(\lambda_k) > 0 = \Phi(\lambda^\star)$, we have $\lambda_{k+1} > \lambda_k$; thus the sequence $\{\lambda_k\}$ is monotonically increasing.
Since $F(\mathbf{x}) \le \lambda^\star$ for all $\mathbf{x}\in \mathcal{D}$,
we have $\lambda_{k+1} \leq \lambda^\star$,
so the sequence is bounded above and hence convergent to some \(\bar{\lambda}\le\lambda^\star\).
Taking limits and using continuity of \(\Phi\). Since $I > 0$ and lower semicontinuous,
\[
0 = \lim_{k\to\infty} (\lambda_{k+1} - \lambda_k)
   = \lim_{k\to\infty} \frac{\Phi(\lambda_k)}{I(\mathbf{x}_k)}
   \implies \Phi(\bar{\lambda}) = 0.
\]
By the uniqueness of the zero of \(\Phi\), it follows that \(\bar{\lambda} = \lambda^\star\).
Consequently, \(\mathbf{x}_k\) converges to an optimal solution 
\(\mathbf{x}^\star \in \arg\max_{\mathbf{x}\in\mathcal{D}} F(\mathbf{x})\).
\end{proof}

\bigskip

\begin{corollary}[Finite-Step Termination for Discrete Feasible Sets]
\label{thm:finite-step}
Suppose the feasible set \(\mathcal{D}\) in Theorem~\ref{thm:convergence} is finite and \(I(\mathbf{x}) > 0\) for all \(\mathbf{x} \in \mathcal{D}\).
Let
\[
\mathcal{R} = 
\Big\{ \frac{M(\mathbf{x})}{I(\mathbf{x})} : \mathbf{x} \in \mathcal{D} \Big\},
\qquad
\lambda^\star = \max \mathcal{R},
\]
then the Dinkelbach algorithm terminates in a finite number of iterations,
returning a globally optimal pair 
\((\mathbf{x}^\star, \lambda^\star)\) satisfying 
\(F(\mathbf{x}^\star) = \lambda^\star\).
If \(|\mathcal{R}| = r\), at most \(r\) iterations are required.
\end{corollary}

\begin{proof}
Each \(\lambda_k\) equals the ratio \(M(\mathbf{x}_k)/I(\mathbf{x}_k)\)
for some feasible \(\mathbf{x}_k\), so \(\lambda_k \in \mathcal{R}\). Because $\mathcal{D}$ is finite, the set \(\mathcal{R}\) is also finite. According to Theorem~\ref{thm:convergence}, $\{\lambda_k\}$ is a monotonically increasing series and converges to $\lambda^{\star}$. Hence the sequence forms a strictly increasing chain of distinct elements
in the finite set \(\mathcal{R}\),
and thus must reach the maximal element \(\lambda^\star\) 
in at most $r$ steps.
\end{proof}




\bibliographystyle{elsarticle-num-names}
\bibliography{sample}

@article{avci2017multi,
  title={A Multi-objective, simulation-based optimization framework for supply chains with premium freights},
  author={Avci, Mualla Gonca and Selim, Hasan},
  journal={Expert Systems with Applications},
  volume={67},
  pages={95--106},
  year={2017},
  publisher={Elsevier}
}

@article{bertsimas2006robust,
  title={A robust optimization approach to inventory theory},
  author={Bertsimas, Dimitris and Thiele, Aur{\'e}lie},
  journal={Operations research},
  volume={54},
  number={1},
  pages={150--168},
  year={2006},
  publisher={INFORMS}
}

@article{charnes1962programming,
  title={Programming with linear fractional functionals},
  author={Charnes, Abraham and Cooper, William W},
  journal={Naval Research logistics quarterly},
  volume={9},
  number={3-4},
  pages={181--186},
  year={1962},
  publisher={Wiley Online Library}
}

@misc{cpsatlp,
  title        = {CP-SAT},
  version      = {v9.12},
  author       = {Laurent Perron and Fr\'ed\'eric Didier},
  organization = {Google},
  url          = {https://developers.google.com/optimization/cp/cp_solver/},
  year         = {2025}
}

@article{cruz2024integrated,
  title={An integrated production planning and inventory management problem for a perishable product: optimization and Monte Carlo simulation as a tool for planning in scenarios with uncertain demands},
  author={Cruz, Jeferson Auto da and Salles-Neto, Luiz Leduino de and Schenekemberg, Cleder Marcos},
  journal={Top},
  volume={32},
  number={2},
  pages={263--303},
  year={2024},
  publisher={Springer}
}

@article{dinkelbach1967nonlinear,
  title={On nonlinear fractional programming},
  author={Dinkelbach, Werner},
  journal={Management science},
  volume={13},
  number={7},
  pages={492--498},
  year={1967},
  publisher={INFORMS}
}

@article{do2022metamodel,
  title={Metamodel-based simulation optimization: A systematic literature review},
  author={do Amaral, Jo{\~a}o Victor Soares and Montevechi, Jos{\'e} Arnaldo Barra and de Carvalho Miranda, Rafael and de Sousa Junior, Wilson Trigueiro},
  journal={Simulation Modelling Practice and Theory},
  volume={114},
  pages={102403},
  year={2022},
  publisher={Elsevier}
}

@article{fisher1981lagrangian,
  title={The Lagrangian relaxation method for solving integer programming problems},
  author={Fisher, Marshall L},
  journal={Management science},
  volume={27},
  number={1},
  pages={1--18},
  year={1981},
  publisher={INFORMS}
}

@book{gokce2002optimization,
  title={Optimization of sourcing decisions in supply chains},
  author={Gokce, Mahmut Ali},
  year={2002},
  publisher={North Carolina State University}
}

@article{gonccalves2020operations,
  title={Operations research models and methods for safety stock determination: A review},
  author={Gon{\c{c}}alves, Jo{\~a}o NC and Carvalho, M Sameiro and Cortez, Paulo},
  journal={Operations research perspectives},
  volume={7},
  pages={100164},
  year={2020},
  publisher={Elsevier}
}

@incollection{hong2021surrogate,
  title={Surrogate-based simulation optimization},
  author={Hong, L Jeff and Zhang, Xiaowei},
  booktitle={Tutorials in Operations Research: Emerging Optimization Methods and Modeling Techniques with Applications},
  pages={287--311},
  year={2021},
  publisher={INFORMS}
}

@article{maitra2024inventory,
  title={Inventory Management Under Stochastic Demand: A Simulation-Optimization Approach},
  author={Maitra, Sarit},
  journal={arXiv preprint arXiv:2406.19425},
  year={2024}
}

@article{miller2022optimization,
  title={On the optimization of benefit to cost ratios for public sector decision making},
  author={Miller, Frederick and Kaya, Yaren Bilge and Dimas, Geri L and Konrad, Renata and Maass, Kayse Lee and Trapp, Andrew C and others},
  journal={arXiv preprint arXiv:2212.04534},
  year={2022}
}

@article{Mitchell2011PuLPAL,
  title={PuLP: A Linear Programming Toolkit for Python},
  author={Stuart Mitchell},
  journal={Optimization Online},
  year={2011},
  url={https://optimization-online.org/wp-content/uploads/2011/09/3178.pdf}
}

@article{ogura2022bayesian,
  title={Bayesian optimization methods for inventory control with agent-based supply-chain simulator},
  author={Ogura, Takahiro and Wang, Haiyan and Wang, Qiyao and Kiuchi, Atsuki and Gupta, Chetan and Uchihira, Naoshi},
  journal={IEICE Transactions on Fundamentals of Electronics, Communications and Computer Sciences},
  volume={105},
  number={9},
  pages={1348--1357},
  year={2022},
  publisher={The Institute of Electronics, Information and Communication Engineers}
}

@article{pitty2008decision,
  title={Decision support for integrated refinery supply chains: Part 1. Dynamic simulation},
  author={Pitty, Suresh S and Li, Wenkai and Adhitya, Arief and Srinivasan, Rajagopalan and Karimi, Iftekhar A},
  journal={Computers \& Chemical Engineering},
  volume={32},
  number={11},
  pages={2767--2786},
  year={2008},
  publisher={Elsevier}
}

@article{sharifnia2021robust,
  title={Robust simulation optimization for supply chain problem under uncertainty via neural network metamodeling},
  author={Sharifnia, Seyed Mohammad Ebrahim and Biyouki, Sajjad Amrollahi and Sawhney, Rupy and Hwangbo, Hoon},
  journal={Computers \& Industrial Engineering},
  volume={162},
  pages={107693},
  year={2021},
  publisher={Elsevier}
}

@book{smith2013handbook,
  title={Handbook of stochastic models and analysis of manufacturing system operations},
  author={Smith, J MacGregor and Tan, Bar{\i}{\c{s}}},
  volume={20013},
  year={2013},
  publisher={Springer}
}

@article{tsai2017simulation,
  title={A simulation-based multi-objective optimization framework: A case study on inventory management},
  author={Tsai, Shing Chih and Chen, Sin Ting},
  journal={Omega},
  volume={70},
  pages={148--159},
  year={2017},
  publisher={Elsevier}
}

@book{vandeput2020inventory,
  title={Inventory optimization: Models and simulations},
  author={Vandeput, Nicolas},
  year={2020},
  publisher={Walter de Gruyter GmbH \& Co KG}
}







\end{document}